\documentclass[reqno]{amsart}
\usepackage{amssymb,latexsym}
\usepackage{amsfonts,mathrsfs}
\usepackage[margin=1.4in]{geometry}
\usepackage{amsmath}
\usepackage{enumerate}
\usepackage{mathtools}

\newtheorem{thm}{Theorem}[section]

\newtheorem{prop}[thm]{Proposition}

\newtheorem{definition}[thm]{Definition}

\theoremstyle{remark}

\begin{document}
\bibliographystyle{abbrv}
\title{ The positive mass theorem for non-spin manifolds with distributional curvature }
\author{Yuqiao Li}
\address{University of Science and Technology of China, No. 96, JinZhai Road Baohe District, Hefei, Anhui, 230026, P.R.China.} \email{lyq112@mail.ustc.edu.cn}
\thanks {2010 Mathematics Subject Classification 53C44 83C99}
\thanks{The research is supported by the National Nature Science Foundation of China No. 11721101 No. 11526212}

\begin{abstract}
  We prove the positive mass theorem for manifolds with distributional curvature which have been studied in \cite{Lee2015} without spin condition. In our case, the manifold $M$ has asymptotically flat metric $g\in C^0\bigcap W^{1,p}_{-q}$, $p>n$, $q>\frac{n-2}{2}$. We show that the generalized ADM mass $m_{ADM}(M,g)$ is non-negative as long as $q=n-2$, and $g$ has non-negative distributional scalar curvature, bounded curvature in the Alexandrov sense with its distributional Ricci curvature belonging to certain weighted Lebesgue space and some extra conditions.
\end{abstract}

\maketitle
\numberwithin{equation}{section}

\section{Introduction}
Suppose that $(M,g)$ is an asymptotically flat n-manifold.
The positive mass theorem says that if the scalar curvature of $(M,g)$ is integrable and non-negative, then the ADM mass of $(M,g)$ is non-negative.
This was proved by Schoen and Yau \cite{schoen1979} under the condition of $n<8$ and by Witten \cite{Witten1981A} for spin manifolds of any dimension.
Bartnik \cite{Bartnik1986} showed that the mass is independent of the choice of the coordinate at infinity provided that the metric $g\in W^{2,p}_{-q}$, $p>n$, $q\geq\frac{n-2}{2}$ with $R(g)\in L^1$.
Miao \cite{Miao2002a} established the positive mass theorem on manifolds with corners along a hypersurface $\Sigma$ adding the condition which was seen as a replacement of the non-negativity of the scalar curvature that the mean curvature of the hypersurface in the compact part is larger than or equal to the mean curvature of the hypersurface in the noncompact part. He constructed smoothings $g_{\epsilon}$ of $g$ in a tubular neighborhood of $\Sigma$ and then used conformal change of $g_{\epsilon}$ to make the scalar curvature non-negative. McFeron and Szekelyhidi \cite{McFeron2012} proved Miao's theorem by Ricci flow to smooth the metric.

Recently, Lee and LeFloch \cite{Lee2015} generalized the scalar curvature and the ADM mass to distributional sense and proved a positive mass theorem for spin $n-$manifolds with $g\in C^0\bigcap W^{1,n}_{loc}$.  They generalized Witten's arguments on spin manifolds and formulate the Lichnerowicz Weitzenbock indentity in the distributional sense. Their main theorem is

\begin{thm}\label{generalizedmass}
   Let $M$ be a smooth spin $n$-manifold($n\geq 3$) with an asymptotically flat metric $g\in C^0\bigcap W_{-q}^{1,n}$, $q> \frac{n-2}{2}$, if the distributional scalar curvature $\ll R_g,u \gg\geq 0$ for every compactly supported smooth non-negative function $u$, then its generalized ADM mass $m_{ADM}(M,g)$ is non-negative, that is,
  \[ m_{ADM}(M,g)\geq 0 \]
  Moreover, $m_{ADM}(M,g)=0$ if and only if $(M,g)$ is isometric to Euclidean space.
\end{thm}

We consider the problem without the assumption that $M$ is spin. Certainly, we have to add some extra conditions on the geometry of $(M,g)$ because the proof of Schoen and Yau for manifolds with $n<8$ is hard to generalize in weak regularity. Our approach is to construct smoothings $g_{\epsilon}$ of the metric and use Ricci flow to ensure the positivity of the scalar curvature. Then, we use classical positive mass theorem on the smoothings to get the mass of $g_{\epsilon}$ is non-negative.
Finally, we get the mass of $g_{\epsilon}$ converges to that of $g$ as $\epsilon$ tends to zero under some assumptions and thus obtain the mass of $g$ is also non-negative. Our main theorem is the following.

\begin{thm}\label{mainth}
   Let $M$ be a smooth $n$-manifold $(3\leq n\leq 7)$ with an asymptotically flat metric $g\in C^0\bigcap W^{1,p}_{-q}$, $p>n$, $q>\frac{n-2}{2}$. Assume that
  \begin{enumerate}[(i)]
    \item $g$ has bounded curvature $Rm(g)$
    \[ C'\leq Rm(g) \leq C \]
    in the sense of Aleksandrov, and the set $\Omega=\{ x\in M:Rm(g)(x)\leq0 \}$ is compact;
    \item $g$ has non-negative scalar curvature distribution $R_g$
    \[ \ll R_g,u \gg\geq 0 \]
    for any smooth compactly supported non-negative function $u$ on $M$;
    \item the Ricci curvature distribution $\ll R_{ij},u \gg$ is in $L^p_{-q-2}$;
     \item
    \[ \lim_{\delta\rightarrow0}\frac{1}{\delta}\int_{0<l(x)<\delta}V\cdot\bar{\nabla}ld\mu_{h}=0 \]
    where $l(x)=dist(x,\Omega)$, for any smooth background metric $h$, $\cdot$ and $\bar{\nabla}$ are inner product and Levi-Civita connection with respect to $h$, $V$ is as in the definition of scalar curvature distribution;
    \item $R_g$ is a finite, signed measure outside a compact set;
    \item $q=n-2$.
  \end{enumerate}
Then, its generalized ADM mass $m_{ADM}(M,g)$ is non-negative. Moreover, $m_{ADM}(M,g)=0$ if and only if $(M,g)$ is isometric to the Euclidean space.
\end{thm}

Briefly speaking, we will construct smooghings $g_{\epsilon}$ that approaches $g$ and these conditions guarantee that $g_{\epsilon}$ nearly satisfies the assumptions of the positive mass theorem. We will see in the following discussion that $(i),(iii)$  imply $g_{\epsilon}\in W^{2,p}_{-q}$; $(ii),(iv)$ ensure that $R_{g_{\epsilon}}$ is almost nonnegative; $(v),(vi)$ connect the mass of $g_{\epsilon}$ to that of $g$.

In section 2 we will state some basic definitions and properties about ADM mass and generalized ADM mass. Then a brief description of the Ricci flow and mass under the Ricci flow is discussed. In section 3, we construct the smoothings $g_{\epsilon}$ and we can see where the extra conditions are used from the construction. In section 4, we discuss the change of scalar curvature after smoothing. In section 5, the change of mass is discussed and we get the convergence of mass. In section 6, we prove the main theorem.

\hspace{0.4cm}

\noindent {\bf Acknowledgements}. I would like to express my gratitude to my advisor Professor Jiayu Li. He gave me so much useful suggestions and inspired me to complete this work.

\section{Preliminaries}

\subsection{The ADM mass}

We first recall several definitions about weighted Sobolev spaces and asymptotically flat manifolds in \cite{Bartnik1986}. Let $r=|x|$, $\sigma=(1+r^2)^{1/2}$ for $x\in\mathbb{R}^n$, $n\geq 3$.

\begin{definition}
  The weighted Lebesgue space $L^p_{\delta}$, $1\leq p\leq \infty$, with weight $\delta\in\mathbb{R}$ is the space of all functions $u$ in $L^p_{loc}(\mathbb{R}^n)$ such that the following $L^p_{\delta}$ norm is finite
  \begin{equation*}
    \|u\|_{L^p_{\delta}}= \begin{dcases}
                            \left( \int_{\mathbb{R}^n}|u|^p\sigma^{-\delta p-n}dx \right)^{1/p}, & p<\infty, \\
                            ess\sup_{\mathbb{R}^n}(\sigma^{-\delta}|u|), & p=\infty.
                          \end{dcases}
  \end{equation*}
   The weighted Sobolev space $W^{k,p}_{\delta}$ is the space of all functions $u$ with finite norm
   \[ \|u\|_{W^{k,p}_{\delta}}=\sum^k_{j=0} \|D^ju\|_{L^p_{\delta-j}} \]
\end{definition}

\begin{definition}
  A smooth n-manifold $(M,g)$ with complete Riemannian metric $g\in W^{1,q}_{loc}$ for some $n<q<\infty$ is said to be asymptotically flat if there is a compact set $K\subset\subset M$ such that there is a diffeomorphism $\Phi : M\backslash K\rightarrow\mathbb{R}^n\backslash B_R$ where $B_R$ is a closed ball in $\mathbb{R}^n$ satisfies
  \[ (\Phi_\ast g)_{ij}-\delta_{ij}\in W^{1,q}_{-\tau},\  \text{for some decay rate $\tau>0$} \]
\end{definition}

We may consider $\Phi $ as the coordinates of $M$ at infinity, so we will usually write $g\in W^{1,q}_{-\tau}(M)$ to denote that $(M,g)$ is asymptotically flat. Fix a metric $h$ on $M$ which is the Euclidean metric outside $K$, and let $r$ be the smooth function on $M$ such that $r=|x|$ outside $\Phi^{-1}(B_2)$ and $r=1$ on $\Phi^{-1}(B_1)$. Then we can also define the weighted Holder norm $C^{k,\alpha}_{\delta}$ as
\[ \|u\|_{C^{k,\alpha}_{\delta}}=\sum_{0}^{k}\sup_M r^{-\delta+j}|\nabla^ju|+\sup_{x,y\in M}\left( \min(r(x),r(y))^{-\delta+k+\alpha}\frac{|\nabla^ku(x)-\nabla^ku(y)|}{|x-y|^{\alpha}}\right) \]
where the derivatives and norms are taken with respect to $h$.

If $g\in C_{-\delta}^{1,\alpha}$ for $\delta>\frac{n-2}{2}$ and the scalar curvature $R(g)\in L^1$, then the ADM mass is definded by
\[ m(g)=\frac{1}{c(n)}\lim_{r\rightarrow\infty}\int_{\partial B_r}(g_{ij,j}-g_{jj,i})\text{d}S^i \]
where the derivatives are taken with respect to the Euclidean metric. Bartnik showed that the mass is a geometric invariant under the above asymptotic decay condition, that is the mass is independent of the choice of the coordinate. The idea is based on the expression of the scalar curvature of $(M,g)$ in local coordinates, and the scalar curvature can be written as
\begin{equation*}
  \begin{split}
     R(g)= & |g|^{-\frac12}\partial_i(|g|^{\frac12}g^{ij}(\Gamma_j-\frac12\partial_j(\log|g|))) \\
       & -\frac12g^{ij}\Gamma_i\partial_j(\log|g|)+g^{ij}\Gamma_{ik}^l\Gamma_{jl}^k
  \end{split}
\end{equation*}
where $\Gamma^i=g^{kl}\Gamma_{kl}^i$. For asymptotically flat metric in $C_{-\delta}^{1,\alpha}$ for $\delta>\frac{n-2}{2}$, we have
\[ |g|^{\frac12}g^{ij}(\Gamma_j-\frac12\partial_j(\log|g|))=g_{ij,j}-g_{jj,i}+\mathcal{O}(r^{-1-2\delta}) \]
\begin{equation*}
  \begin{split}
    c(n)m(g)= &\int_{M\backslash B_r}R(g)\text{d}V+\int_{\partial B_r}|g|^{\frac{1}{2}}g^{ij}(\Gamma_j-\frac{1}{2}\partial_j(\log|g|))dS^i\\
    &-\int_{M\backslash B_r}\left(\frac12g^{ij}\Gamma_i\partial_j(\log |g|) -g^{ij}\Gamma_{ik}^l\Gamma_{jl}^k\right)\text{d}V \\
       =& \int_{M\backslash B_r}R(g)\text{d}V+\int_{\partial B_r}(g_{ij,j}-g_{jj,i})dS^i+\mathcal{O}(r^{-\lambda})
  \end{split}
\end{equation*}
for some $\lambda>0$. The positive mass theorem is the following.

\begin{thm}\label{massthm}
  If $M$ is spin or $n\leq 7$, $g\in C_{-\delta}^{1,\alpha}$, $\delta>\frac{n-2}{2}$, $\alpha\in (0,1)$, $R(g)\geq 0$, $R(g)\in L^1(M)$, then $m(g)\geq 0$, and equality holds only when $(M,g)$ is isometric the to Euclidean space.
\end{thm}

\subsection{The generalized ADM mass}

In this paper, we are given a smooth $n$-manifold $M$ with $n\geq 3$ endowed with a fixed smooth background metric denoted by $h$.
Let us describe the definitions of generalized scalar curvature and the generalized ADM mass originated from \cite{Lee2015}.

\begin{definition}\label{scalar}
Let $(M,h)$ be a smooth Riemannian manifold, $h$ be the smooth background metric.Given any Riemannian metric $g\in L^{\infty}_{loc}\bigcap W^{1,2}_{loc}$ on $M$ with
$g^{-1}\in L^{\infty}_{loc}$, for any compactly supported smooth function $u:M\rightarrow \mathbb{R} $, the scalar curvature distribution $R_g$ is defined by
\[ \ll R_g,u \gg :=\int_{M}\left(-V\cdot \bar{\nabla}(u\frac{\text{d}\mu_g}{\text{d}\mu_h})+Fu\frac{\text{d}\mu_g}{\text{d}\mu_h}\right)\text{d}\mu_h \]
where
\[ \Gamma_{ij}^k=\frac12g^{kl}(\bar{\nabla}_ig_{jl}+\bar{\nabla}_jg_{il}-\bar{\nabla}_lg_{ij})=\Gamma_{ij}^k(g)-\bar{\Gamma}_{ij}^k \]
\[ V^k=g^{ij}\Gamma_{ij}^k-g^{ik}\Gamma_{ji}^j=g^{ij}g^{kl}(\bar{\nabla}_j g_{il}-\bar{\nabla}_l g_{ij}) \]
\[ F=\bar{R}-\bar{\nabla}_k g^{ij}\Gamma_{ij}^k+\bar{\nabla}_kg^{ik}\Gamma_{ji}^j+g^{ij}(\Gamma_{kl}^k\Gamma_{ij}^l-\Gamma_{jl}^k\Gamma_{ik}^l) \]
and $\bar{\nabla}$,$\bar{R}$ are the Levi-Civita connection and scalar curvature with respect to $h$, $\Gamma_{ij}^k(g)$ and $\bar{\Gamma}_{ij}^k$ denote the Christoffel symbols of $g$ and $h$ respectively, the dot product is taken using the metric $h$, and $\text{d}\mu_h$ and $\text{d}\mu_g$ denote the volume measures with $h$ and $g$, respectively.
\end{definition}

In the case of $g\in C^2$, the scalar curvature distribution $R_g$ is well-defined in the usual way and is a continuous function; in this case, $\ll R_g,u \gg=\int_M R_gud\mu_g$. And $\ll R_g,u\gg$ does not depend on the choice of the background metric $h$, as long as $g\in C^0\bigcap W^{1,2}_{loc}$.
$R_g$ is said to be non-negative when $\ll R_g,u \gg\geq 0$ for every non-negative test function $u$.
In \cite{Lee2015}, they also showed that if $g\in C^0\bigcap W^{1,n}_{loc}$, then $\ll R_g,u \gg$ makes sense for all compactly supported functions $u\in L^{\frac{n}{n-2}}$ with $\bar{\nabla}u\in L^{\frac{n}{n-1}}$.

Now let $M$ be a smooth n-dimensional asymptotically flat manifold such that there is a compact set $K\subset M$ and a diffeomorphism $\Phi$ between $M\backslash K$ and $\mathbb{R}^n\backslash B_1(0)$, where $B_1(0)$ denotes the unit ball in $\mathbb{R}^n$. Choose any smooth background metric $h$ on $M$ such that $h_{ij}=\delta_{ij}$ in the coordinate chart $M\backslash K\cong\mathbb{R}^n\backslash B_1(0)$ determined by $\Phi$. We also choose a smooth positive function $r$ on $M$ that is the radial coordinate on $M\backslash K\cong\mathbb{R}^n\backslash B_1(0)$ and is less than $2$ on $K$.

\begin{definition}
  Let $M,\Phi,h,r$ be as above. For any $p\geq 1$ and $q>0$ , a $L^{\infty}_{loc}$ Riemannian metric $g$ on $M$ with $L^{\infty}_{loc}$ inverse is $W^{k,p}_{-q}$, then the generalized ADM mass of $(M,g)$ is defined as
  \[ m_{ADM}(M,g):=\frac{1}{2(n-1)w_{n-1}}\inf_{\epsilon>0}\liminf_{\rho\rightarrow +\infty}\left(\frac{1}{\epsilon}\int_{\rho<r<\rho+\epsilon}V\cdot\bar{\nabla}r\text{d}\mu_h\right) \]
where $V$ is the vector field in the definition of scalar curvature distribution and $w_{n-1}$ is the volume of the standard unit $(n-1)$-sphere.
\end{definition}

From the definition of the distributional ADM mass, if $g\in W^{2,p}_{-q}$ for $p>n$, $q> \frac{n-2}{2}$ and its scalar curvature is integrable, then the generalized ADM mass is equivalent to the usual definition of ADM mass. So $m_{ADM}$ is actually a generalization of the usual ADM mass.

\subsection{The Ricci flow and mass along Ricci flow}

The Ricci flow is a family of metrics $g(t)$ on a Riemannian manifold $M$ satisfying the equation
\[ \frac{\partial}{\partial t}g=-2Ric \]
where $Ric$ is the Ricci tensor of the time-dependent metric $g(t)$.
For any $C^{\infty}$ metric $g_0$ on a closed manifold $M^n$, there exists a unique solution $g(t)$, $t\in [0,\epsilon)$, to the Ricci flow equation for some $\epsilon>0$, with $g(0)=g_0$ \cite{Hamil}. Because of the diffeomorphism invariance of the Ricci tenser, the Ricci flow equation is only weakly parabolic, so we have to use the DeTurck's trick to get an equivalent flow which is strictly parabolic. Given a fixed background connection $\tilde{\Gamma}$ which is assumed to be the Levi-Civita connection of a metric $\tilde{g}$, we define the Ricci-DeTurck flow by
\[ \frac{\partial}{\partial t}g_{ij}=-2R_{ij}+\nabla_iW_j+\nabla_jW_i \]
\[ g(0)=g_0 \]
where the time-dependent 1-form $W=W(t)$ is defined by
\[ W_j=g_{jk}g^{pq}(\Gamma^k_{pq})-\tilde{\Gamma}^k_{pq}. \]

Shi \cite{Shi1989} also showed the short existence of the Ricci flow on noncompact manifold with bounded curvature. Given an asymptotically flat n-manifold $(M,g)$, $g\in C^2_{\delta}$, assume that $g$ has uniformly bounded curvature, then, by the result of Shi, there is a Ricci flow $g(t)$ with $g(0)=g$ in a short time interval.  Dai-Ma \cite{Dai2007} proved that along the Ricci flow, the metric remains asymptotically flat of the same order, i.e. $g(t)\in C^2_{\delta}$,  by the maximum principle of Ecker-Huisken \cite{Ecker1991}. Li \cite{Li2018} recently also showed that the asymptotically flat condition is preserved under Ricci flow and if the initial metric $g(0)\in C^2_{\delta}$ with $\delta>\frac{n-2}{2}$ and $R\in L^1$, then the mass is unchanged. He also showed that the long-time existence of the Ricci flow on an asymptotically flat 3-manifold with nonnegative scalar curvature will imply the positive mass theorem.

\section{Construction of smoothings}

Let $M$ be an asymptotically flat $n$-manifold $(n\geq 3)$ with background metric $h$. Assume that $g\in C^0\bigcap W^{1,p}_{-q}$, $p>n$, $q>\frac{n-2}{2}$ is a Riemannian metric on $M$, then by the Sobolev inequality of weighted Sobolev space from \cite{Bartnik1986}, we have
\[ \| g \|_{\infty,-q}\leq C\| g \|_{1,p,-q} \]
for some constant $C$ independent of $q$, so $g\in L^{\infty}_{-q}$. That is $|g|=o(r^{-q})$ as $r\rightarrow\infty$ by the definition of the weighted Sobolev space. Also, from the following inequality, for $0<\alpha \leq 1-\frac{n}{p}<1$, then
\[ \| g \|_{C^{0,\alpha}_{-q}}\leq C\| g \|_{1,p,-q} \]
so $g\in C^{0,\alpha}_{-q}$, that is $\| g \|_{C^{0,\alpha}_{-q}(M\backslash B_R)}=o(1)$ as $R\rightarrow\infty$, where $B_R$ is a geodesic ball of radius $R$ in $M$.

Let us recall some facts about complete manifold with bounded curvature.
From \cite{Nikolaev1983}, we know that such metrics are locally $C^{1,\alpha}$.
\begin{thm}\label{loc}
  Let $M$ be a space with bounded curvature. Then in a neighborhood of each point, we can introduce a harmonic coordinate system. The components $g_{ij}$ of the metric tensor in any harmonic coordinate system in $M$ are continuous functions of $W^{2,p}(\Omega)$ for any $p\geq 1$ where $\Omega\subset\mathbb{R}^n$ is a domain of harmonic coordinate.
\end{thm}
Thus, if we assume that $g\in C^0\bigcap W^{1,p}_{-q}$ with bounded curvature, then $g\in W^{2,p}_{loc}$.
From \cite{Simon2002}, we have the following theorem with the assumption of bounded curvature.
\begin{thm}\label{bdd}
  Let $g\in C^0$ be a metric with bounded curvature on a manifold $M$, with curvature $K(g)$
  \[ C'\leq K(g) \leq C \]
  in the sense of Aleksandrov. We may approximate $g$ by smooth Riemannian metrics $g_{\alpha}$, $\alpha\in\mathbb{N}$ such that
  \[ C'-\frac{1}{\alpha}\leq K(g_{\alpha})\leq C+\frac{1}{\alpha} ,\]
  \[ \lim_{\alpha\rightarrow\infty}|g_{\alpha}-g|_{C^{1,\beta}(\Omega)}\rightarrow0,  \]
  \[ \lim_{\alpha\rightarrow\infty}|g_{\alpha}-g|_{C^0(M)}\rightarrow0, \]
  for open $\Omega\subset M$ whose closure is compact. Furthermore if the curvature satisfies
  \[ B'g\leq Ricci(g)\leq Bg \]
  then
  \[ (B'-\frac{1}{\alpha})g_{\alpha}\leq Ricci(g_{\alpha})\leq (B+\frac{1}{\alpha})g_{\alpha}. \]
\end{thm}

Let $g_{\epsilon_i}$, $i\in\mathbb{N}$, $\epsilon_i\rightarrow0$ as $i\rightarrow\infty$ be the smooth Riemaiannian metrics in Theorem \ref{bdd} that approximate $g$.
We will omit $i$ for convenience. Then, $g_{\epsilon}\in W^{1,p}_{-q}$ for $\epsilon$ small and satisfy
\[\frac{1}{1+C_1(\epsilon)}g_{\epsilon ij}\leq g_{ij}\leq (1+C_1(\epsilon))g_{\epsilon ij}\]
 \[\frac{1}{1+C_1(\epsilon)}g^{ij}_{\epsilon}\leq g^{ij}\leq (1+C_1(\epsilon))g^{ij}_{\epsilon}\]
\[ \frac{1}{1+C_1(\epsilon)}\partial g_{\epsilon ij}\leq \partial g_{ij}\leq (1+C_1(\epsilon))\partial g_{\epsilon ij} \]
\[ \frac{1}{1+C_1(\epsilon)}\partial g^{ij}_{\epsilon}\leq \partial g^{ij}\leq (1+C_1(\epsilon))\partial g^{ij}_{\epsilon}. \]

Similar to the definition of the scalar curvature distribution, we can also define the Ricci curvature distribution.
\begin{definition}\label{Ricci}
  Let $(M,h)$ be a smooth Riemannian $n$-manifold with smooth background metric $h$. Given any metric $g\in L^{\infty}_{loc}\bigcap W^{1,2}_{loc}$ on $M$ with
  $g^{-1}\in L^{\infty}_{loc}$, for any compactly supported smooth function $u: M\rightarrow\mathbb{R}$, the Ricci curvature distribution $R_{ij}(g)$ is defined by
  \[ \ll R_{ij}(g),u\gg=\int_{M}(-\Gamma^k_{ij}\cdot \bar{\nabla}_k(u\frac{d\mu_g}{d\mu_h})+\Gamma^k_{ik}\bar{\nabla}_j(u\frac{d\mu_g}{d\mu_h}))d\mu_h+\int_MGud\mu_g \]
  where
  \[ G=R_{ij}(h)+\Gamma^k_{ij}\Gamma^l_{kl}-\Gamma^k_{il}\Gamma^l_{kj} \]
  \[ \Gamma^k_{ij}=\Gamma^k_{ij}(g)-\Gamma^k_{ij}(h) \]
  and $\bar{\nabla}$,$\bar{R}$ are the Levi-Civita connection and scalar curvature with respect to $h$, $\Gamma_{ij}^k(g)$ and $\bar{\Gamma}_{ij}^k$ denote the Christoffel symbols of $g$ and $h$ respectively, the dot product is taken using the metric $h$, and $\text{d}\mu_h$ and $\text{d}\mu_g$ denote the volume measures associated with $h$ and $g$, respectively.
\end{definition}
We then define the weighted Lebesgue spaces $L^p_{\delta}$ in the distributional sense.
\begin{definition}
  The distributional Ricci curvature $\ll R_{ij}(g),u\gg$ is said to be $L^p_{\delta}$ if its norm
  \begin{equation}\label{Ric}
  \begin{split}
     \ll R_{ij}(g),u \gg_{L^p_{\delta}}= & \int_M |R_{ij}(h)|^p r^{-\delta p-n}ud\mu_g
     +\int_M(|\Gamma^k_{ij}\Gamma^l_{kl}|^{\frac{p}{2}}+|\Gamma^k_{il}\Gamma^l_{kj}|^{\frac{p}{2}})r^{(-\delta-1)p-n}ud\mu_{g} \\
       & + \int_M|\Gamma^k_{ij}|^pr^{(-\delta-1)p-n}|\bar{\nabla}_k(u\frac{d\mu_g}{d\mu_h})|d\mu_h
       + \int_M|\Gamma^k_{ik}|^pr^{(-\delta-1)p-n}|\bar{\nabla}_j(u\frac{d\mu_g}{d\mu_h})|d\mu_h
  \end{split}
  \end{equation}
  is finite for any compactly supported smooth function $u$ on $M$.
\end{definition}

\begin{prop}\label{u}
  Let $M$ be a smooth manifold with a smooth background metric $h$. Given any Riemannian metric $g\in C^0\bigcap W^{1,p}_{loc}$, then the scalar curvature distribution $R_g$ in the sense of Definition \ref{scalar} and the Ricci curvature distribution $R_{ij}(g)$ in the sense of Definition \ref{Ricci} make sense for all compactly supported functions $u\in L^{\frac{p}{p-2}}$ with its derivatives lie in $L^{\frac{p}{p-1}}$.
\end{prop}

\def\Rnum#1{\uppercase\expandafter{\romannumeral #1}}
\begin{prop}\label{Rij}
  If $\ll R_{ij}(g),u\gg\in L^p_{-q-2}$, then $R_{ij}(g_{\epsilon})\in L^p_{-q-2}$.
\end{prop}
\begin{proof}
Given any $\rho>2$, consider the cut-off function
\begin{equation}\label{cutoff}
  v_{\rho}(x)=\begin{cases}
                1, & r(x)\leq\rho \\
                2-\frac{r}{\rho}, & \rho<r(x)\leq2\rho \\
                0, & r(x)>2\rho.
              \end{cases}
\end{equation}
Since $v_{\rho}(x)$ is compactly supported Lipschitz continuous function,
by Proposition \ref{u}, let $u=v_{\rho}\frac{d\mu_{g_{\epsilon}}}{d\mu_g}$,$h=g_{\epsilon}$, $\delta=-q-2$ in \eqref{Ric}, we have
  \begin{equation*}
    \begin{split}
       \ll R_{ij}(g),u\gg_{L^p_{-q-2}}= &\int_M |R_{ij}(g_{\epsilon})|^p r^{(q+2)p-n}v_{\rho}d\mu_{g_{\epsilon}}
     +\int_M(|\Gamma^k_{ij}\Gamma^l_{kl}|^{\frac{p}{2}}+|\Gamma^k_{il}\Gamma^l_{kj}|^{\frac{p}{2}})r^{(q+1) p-n}v_{\rho}d\mu_{g_{\epsilon}} \\
       & + \int_M|\Gamma^k_{ij}|^pr^{(q+1)p-n}|\bar{\nabla}_kv_{\rho}|d\mu_{g_{\epsilon}}
       + \int_M|\Gamma^k_{ik}|^pr^{(q+1)p-n}|\bar{\nabla}_jv_{\rho}|d\mu_{g_{\epsilon}}  \\
        = &\Rnum{1}+\Rnum{2}+\Rnum{3}+\Rnum{4} <\infty.
    \end{split}
  \end{equation*}
Since $g\in W^{1,p}_{-q}$ and $g_{\epsilon}\in W^{1,p}_{-q}$, we have $\Gamma_{ij}^k\in L^p_{-q-1}$ and $\Gamma_{ij}^k\Gamma_{kl}^l\in L^{\frac{p}{2}}_{-2q-2}$.
By the fact that $-2q-2<-n$, then $L^{\frac{p}{2}}_{-2q-2}\subset L^1_{-n}=L^1$, we have $\Rnum{2}<\infty$.
\begin{equation*}
  \begin{split}
      \Rnum{3}= & \int_{\{ \rho<r<2\rho \}}|\Gamma^k_{ij}|^pr^{(q+1)p-n}\frac{1}{\rho}|\bar{\nabla r}|d\mu_{g_{\epsilon}} \\
       =& \frac{1}{\rho}\int_{\{ \rho<r<2\rho \}}|\Gamma^k_{ij}|^pr^{(q+1)p-n}d\mu_{g_{\epsilon}}\\
       <&\frac{1}{2}\|\Gamma^k_{ij}\|_{L^p_{-q-1}}<\infty
  \end{split}
\end{equation*}
for $\rho>2$.
And so $\Rnum{4}<\infty$. Thus we have $R_{ij}(g_{\epsilon})\in L^p_{-q-2}$.
\end{proof}

We will need the following theorem in \cite{Bartnik1986}.
\begin{thm}\label{reg}
  Suppose $(M,g,\Phi)$ is a structure of infinity with $\Phi_{\ast}g-\delta\in W^{2,p}_{-\eta}(\mathbb{R}^n\backslash B_R)$ for some $\eta>0$, $q>0$, $R\geq 1$,
  the Ricci tensor of $(M,g)$ satisfies
  \[ Ric(g)\in L^p_{-\tau-2}(M) \]
  for some nonexceptional $\tau >\eta$. Then there is a structure of infinity $\Theta$ defined by coordinates harmonic near infinity which satisfies
  $(\Theta_{\ast}g-\delta)\in W^{2,p}_{-\tau}(\mathbb{R}^n\backslash B_{R_1})$, for some $R_1\geq R$.
\end{thm}
Note that $\delta\in\mathbb{R}$ is said to be nonexceptional if $\delta\in\mathbb{R}/\{ k\in\mathbb{Z},k\neq -1,-2,\cdots, 3-n \}$, where the exceptional values
$\{ k\in\mathbb{Z},k\neq -1,-2,\cdots, 3-n \}$ correspond to the orders of growth of harmonic functions in $\mathbb{R}^n\backslash B_1$.

\begin{thm}\label{gepsilon}
  Let $(M,g)$ be an asymptotically flat $n$-manifold $(3\leq n\leq 7)$ with $g\in C^0\bigcap W^{1,p}_{-q}$, $p>n$, $q>\frac{n-2}{2}$. Assume that $g$ has bounded curvature and its distributional Ricci curvature $\ll R_{ij}(g),u \gg\in L^p_{-q-2}$, then there is a series of smoothings $g_{\epsilon}$ such that $g_{\epsilon}\rightarrow g$ in $C^{1,\alpha}_{loc}$ as $\epsilon\rightarrow0$ and $g_{\epsilon}\in W^{2,p}_{-q}$.
\end{thm}
\begin{proof}
  Let $g_{\epsilon}$ as above, then $g_{\epsilon}\rightarrow g$ in $W^{1,p}_{loc}$ as $\epsilon\rightarrow0$ and thus $g_{\epsilon}\in W^{1,p}_{-q}$ as long as $\epsilon$ small. From Theorem \ref{bdd}, we know $g_{\epsilon}\rightarrow g$ in $C^{1,\alpha}_{loc}$ as $\epsilon\rightarrow0$.
  We can apply Theorem \ref{loc} to get $g\in W^{2,p}_{loc}$, so $g_{\epsilon}\in W^{2,p}_{loc}$.
  By the definition of the weighted Sobolev spaces, it is easy to see that $L^p_{\delta}$ is equivalent to the usual Sobolev space when $\delta=-\frac{n}{p}$.
  Therefore, $g_{\epsilon}\in W^{2,p}_{-\frac{n}{p}}$.
  Proposition \ref{Rij} implies that $R_{ij}(g_{\epsilon})\in L^p_{-q-2}$.
  If $\frac{n}{p}<q$, since $\frac{n}{p}>0$, we have $g_{\epsilon}\in W^{2,p}_{-q}$ by Theorem \ref{reg}.
  If $\frac{n}{p}>q$, then $\partial\partial g_{\epsilon}\in L^p_{-\frac{n}{p}-2}\subset L^p_{-q-2}$. Thus, we also get $g_{\epsilon}\in W^{2,p}_{-q}$.
\end{proof}

\section{The change of the scalar curvature}

We now focus on the scalar curvature. Under the hypothesis of Theorem \ref{gepsilon}, let $g_{\epsilon}$ be the smoothings constructed in the previous section.
From the definition of the scalar curvature distribution, we assume
\[ \ll R_g,u \gg\geq 0  \]
for any smooth compactly supported nonnegative function $u$. We choose $g_{\epsilon}$ as the background metric and
\[ u=v\frac{d \mu_{g_{\epsilon}}}{d \mu_g} \]
where $v$ is a smooth compactly supported nonnegative function on $M$. Then according to Definition \ref{scalar},
\begin{equation}\label{positive scalar}
  \begin{split}
     \ll R_g,u \gg =& \int_{M}(-(V\cdot \bar{\nabla}(u\frac{d \mu_g}{d \mu_{g_{\epsilon}}}))+Fu\frac{d \mu_g}{d \mu_{g_{\epsilon}}})d\mu_{g_{\epsilon}} \\
      = &\int_{M} (-g^{ij}g^{kl}(\bar{\nabla}_j g_{il}-\bar{\nabla}_l g_{ij})\bar{\nabla}_k v\\
      &+(R_{\epsilon}-\bar{\nabla}_k g^{ij}\Gamma^k_{ij}+\bar{\nabla}_k g^{ik}\Gamma^j_{ji}+
      g^{ij}(\Gamma^k_{kl}\Gamma^l_{ij}-\Gamma^k_{jl}\Gamma^l_{ik}))v)d\mu_{g_{\epsilon}}\\
      \geq &0
  \end{split}
\end{equation}
where $R_{\epsilon}$ and $\bar{\nabla}$ are the scalar curvature and Levi-Civita connection of $g_{\epsilon}$.
then
\begin{equation}\label{scalar curvature}
  \begin{split}
     \int_M R_{\epsilon}v d\mu_{g_{\epsilon}} & \geq\int_M g^{ij}g^{kl}(\bar{\nabla}_j g_{il}-\bar{\nabla}_l g_{ij})\bar{\nabla}_k v d\mu_{g_{\epsilon}}  \\
       & +\int_M (\bar{\nabla}_k g^{ij}\Gamma^k_{ij}-\bar{\nabla}_k g^{ik}\Gamma^j_{ji}-g^{ij}(\Gamma^k_{kl}\Gamma^l_{ij}-\Gamma^k_{jl}\Gamma^l_{ik}))vd\mu_{g_{\epsilon}}.
  \end{split}
\end{equation}

Let $g\in C^0\bigcap W^{1,p}_{-q}$, $p>n$, $q>\frac{n-2}{2}$, so $g\in C^0\bigcap W^{1,p}_{loc}$,for $p>n$,then by Proposition \ref{u}, the distributional curvature $\ll R_g,u\gg$ makes sense for all compactly supported functions $u\in L^{\frac{p}{p-2}}$ whose derivatives lie in $L^{\frac{p}{p-1}}$.

Now we assume $g$ has bounded curvature in the sense of Aleksandrov and the set $\{ x\in M:Rm(g)(x)\leq0 \}$ is compact, where $Rm(g)(x)$ is the curvature of $g$.
Theorem \ref{bdd} implies that the set $\{ x\in M:Rm(g_{\epsilon})(x)\leq0 \}$ is also compact and so is $\{ x\in M:R_{\epsilon}(x)\leq0 \}$. Let $U=\{ x\in M:R_{\epsilon}(x)\leq 0 \}$, $U_{\delta}=\{ x\in M: dist(x,U)\leq\delta \}$, $l=dist(x,U)$,
\[ v= \begin{cases}
        1, & \mbox{in } U, \\
        1-\frac{l(x)}{\delta}, &\mbox{in } U_{\delta}\backslash U\\
        0, & \mbox{outside of } U_{\delta}.
      \end{cases} \]
From \eqref{scalar curvature}, we have
\begin{equation}\label{R1}
  \int_{U_{\delta}}R_{\epsilon}vd\mu_{g_{\epsilon}}\geq \int_{U_{\delta}}(\bar{\nabla}_k g^{ij}\Gamma^k_{ij}-\bar{\nabla}_k g^{ik}\Gamma^j_{ji}-g^{ij}(\Gamma^k_{kl}\Gamma^l_{ij}-\Gamma^k_{jl}\Gamma^l_{ik}))vd\mu_{g_{\epsilon}}+\int_{U_{\delta}\backslash U}V\cdot\bar{\nabla}vd\mu_{g_{\epsilon}}
=\Rnum{1}+\Rnum{2}
\end{equation}
and
\[ \int_{\{R_{\epsilon}<0\}}R_{\epsilon}d\mu_{g_{\epsilon}}<0 \]
By direct calculation, we get
\begin{equation*}
  \begin{split}
     |\Gamma_{ij}^k|= &|\Gamma_{ij}^k(g)-\Gamma_{ij}^k(g_{\epsilon})|  \\
      = & |\frac{1}{2}g^{kl}(g_{il,j}+g_{jl,i}-g_{ij,l})-\frac{1}{2}g_{\epsilon}^{kl}(g_{\epsilon il,j}+g_{\epsilon jl,i}-g_{\epsilon ij,l})| \\
      = & |\frac{1}{2}g^{kl}(g_{il,j}-g_{\epsilon il,j})+\frac{1}{2}(g^{kl}-g^{kl}_{\epsilon})g_{\epsilon il,j}+\frac{1}{2}g^{kl}(g_{jl,i}-g_{\epsilon jl,i}) \\
       & +\frac{1}{2}(g^{kl}-g^{kl}_{\epsilon})g_{\epsilon jl,i}+\frac{1}{2}(g_{\epsilon}^{kl}-g^{kl})g_{\epsilon ij,l}+\frac{1}{2}g^{kl}(g_{\epsilon ij,l}-g_{ij,l})|\\
      \leq &C_1(\epsilon)|g_{\epsilon}*\partial g_{\epsilon}|
  \end{split}
\end{equation*}
and
\[|\Gamma^j_{kl}(g_{\epsilon})\Gamma^k_{ij}|\leq C_1(\epsilon)Q(g^{-1}_{\epsilon},\partial g_{\epsilon})\]
where $Q$ is a quadratic form of $g^{-1}_{\epsilon}$ and $\partial g_{\epsilon}$. Thus

\begin{equation*}
\begin{split}
   |\bar{\nabla}_k g^{ij}\Gamma^k_{ij}|= & |(\partial_k g^{ij}-\Gamma^i_{kl}(g_{\epsilon})g^{lj}-\Gamma^j_{kl}(g_{\epsilon})g^{il})\Gamma^k_{ij}| \\
     =& |(\partial_k g^{ij}-2\Gamma^i_{kl}(g_{\epsilon})g^{jl})\Gamma^k_{ij}|\\
     \leq & C_1(\epsilon)Q(g^{-1}_{\epsilon},\partial g_{\epsilon}),
\end{split}
\end{equation*}
and
\[ |\Gamma^k_{kl}\Gamma^l_{ij}|\leq C_1(\epsilon)^2Q_1(g^{-1}_{\epsilon},\partial g_{\epsilon}) \]
where $Q_1$ is another quadratic form of $g^{-1}_{\epsilon}$ and $\partial g_{\epsilon}$.

Under the conditions of Theorem \ref{gepsilon}, $g_{\epsilon}\in W^{2,p}_{-q}$, then
 \[ \Gamma^k_{ij}(g_{\epsilon})=O(r^{-q-1}),\]
 \[   g^{kl}_{\epsilon}=1+O(r^{-q}), \]
 \[  g_{\epsilon ij,l}=O(r^{-q-1}),    \]
\[ Q=O(r^{-2q-2}), \]
\[ Q_1=O(r^{-2q-2}). \]
Reture to the scalar curvature, we get
\begin{equation}\label{R2}
  \begin{split}
     \Rnum{1} &\leq \int_{U_{\delta}}(C_1(\epsilon)Q+C_1(\epsilon)^2Q_1)vd\mu_{g_{\epsilon}}\\
     &= C_1(\epsilon)\int_{U_{\delta}}(Q+C_1(\epsilon)Q_1)vd\mu_{g_{\epsilon}}.
  \end{split}
\end{equation}
Observe that the integrand of the above inequality is $O(r^{-2q-2})$, and since $-2q-2< -n$ and by our assumption that the set $\{ x\in M: R_{\epsilon}<0 \}$ is compact, then the integral is finite and
\[ \Rnum{1}\geq -C_1(\epsilon)\]
and
\begin{equation}\label{R3}
  \Rnum{2}=\int_{U_{\delta}\backslash U}V\cdot\bar{\nabla}vd\mu_{g_{\epsilon}}=-\int_{0<l<\delta}\frac{1}{\delta}V\cdot\bar{\nabla}ld\mu_{g_{\epsilon}}.
\end{equation}

Therefore, we can finally prove the following theorem.
\begin{thm}
  Under the conditions of Theorem \ref{gepsilon}, we further assume that
  \[ \ll R_g,u\gg\geq 0 \]
  and the set $\Omega=\{ x\in M: Rm(x)\leq 0 \}$ is compact and satisfies
  \[ \lim_{\delta\rightarrow0}\frac{1}{\delta}\int_{0<l(x)<\delta}V\cdot\bar{\nabla}l d\mu_{h}=0 \]
  where $l(x)=dist(x,\Omega)$, for any smooth background metric $h$, $\cdot$ and $\bar{\nabla}$ are inner product and Levi-Civita connection with respect to $h$, $V$ is as in the definition of scalar curvature distribution.
  Then there is a series of smoothings $g_{\epsilon}\in W^{2,p}_{-q}$, such that
  \[\frac{1}{1+C_1(\epsilon)}g\leq g_{\epsilon}\leq (1+C_1(\epsilon))g\]
   with $\lim_{\epsilon\rightarrow0}C_1(\epsilon)=0$
   and
\[ \int_{\{R_{\epsilon}<0\}}R_{\epsilon}d\mu_{g_{\epsilon}}\geq -C(\epsilon) \]
where $\lim_{\epsilon\rightarrow0}C(\epsilon)=0$.
\end{thm}
\begin{proof}
 From Theorem \ref{gepsilon}, we get $g_{\epsilon}\in W^{2,p}_{-q}$. On $U_{\delta}\backslash U$, we have $R_{\epsilon}>0$, so $R_{\epsilon}v<R_{\epsilon}$ since $v<1$.
 Thus,
\[ \int_{U_{\delta}}R_{\epsilon}vd\mu_{g_{\epsilon}}<\int_{U_{\delta}}R_{\epsilon}d\mu_{g_{\epsilon}}\]
then
\[ \lim_{\delta\rightarrow0}\int_{U_{\delta}}R_{\epsilon}vd\mu_{g_{\epsilon}}\leq\lim_{\delta\rightarrow0}\int_{U_{\delta}}R_{\epsilon}d\mu_{g_{\epsilon}}
=\int_UR_{\epsilon}d\mu_{g_{\epsilon}} \]
    Let $\delta\rightarrow0$ in \eqref{R1}, and from \eqref{R2}, \eqref{R3}, we get
  \[ \int_{\{R_{\epsilon}\leq 0\}}R_{\epsilon}d\mu_{g_{\epsilon}} \geq -C(\epsilon)
  -\lim_{\delta\rightarrow0}\frac{1}{\delta}\int_{U_{\delta}\backslash U}V\cdot\bar{\nabla}ld\mu_{g_{\epsilon}} \]
  Let $\Omega_{\delta}=\{x\in M:dist(x,\Omega)\leq\delta\}$, from our assumption, we have
  \[\lim_{\delta\rightarrow0}\frac{1}{\delta}\int_{\Omega_{\delta}\backslash\Omega}V\cdot\bar{\nabla}ld\mu_{g_{\epsilon}}=0\]
 From Theorem $\ref{bdd}$, there exists $C_2(\epsilon)$ such that for every $\delta$,
 \[ Vol((U_{\delta}\backslash U)\backslash(\Omega_{\delta}\backslash\Omega))<C_2(\epsilon)\delta \]
 with $\lim_{\epsilon\rightarrow0}C_2(\epsilon)=0$.
 Since $g\in W^{1,p}_{-q}$, we have $V\in L^p_{-q-1}$, so $V\in L^1_{loc}$.
  Therefore,
 \begin{equation*}
 \begin{split}
       &\lim_{\delta\rightarrow0}\frac{1}{\delta}\int_{U_{\delta}\backslash U}V\cdot\bar{\nabla}ld\mu_{g_{\epsilon}}\\
       \leq &\lim_{\delta\rightarrow0}\frac{1}{\delta}\int_{\Omega_{\delta}\backslash\Omega}V\cdot\bar{\nabla}ld\mu_{g_{\epsilon}}+
       \lim_{\delta\rightarrow0}\frac{1}{\delta}\int_{(U_{\delta}\backslash U)\backslash(\Omega_{\delta}\backslash\Omega)}V\cdot\bar{\nabla}ld\mu_{g_{\epsilon}}\\
       \leq &C_2(\epsilon)
  \end{split}
 \end{equation*}
 Then we get the result.
\end{proof}

\section{The change of mass}

Now we consider the change of mass after smoothing.
Let $M$ be a smooth asymptotically flat $n$-manifold such that there is a diffeomorphism $\Phi$ between $M\backslash K$ and $\mathbb{R}^n\backslash B_1(0)$, for some compact set $K\in M$, $B_1(0)$ denotes the unit closed ball in $\mathbb{R}^n$. Choose $h$ to be the smooth background metric such that $h_{ij}=\delta_{ij}$ through $\Phi$. Let $r$ be the radial coordinate outside $K$ and less than 2 in $K$.
Let $g\in C^0\bigcap W^{1,p}_{-q}$ be an asymptotically flat metric on $M$ with $p>n$ and $q\geq \frac{n-2}{2}$.
From Theorem 3, there is a series of smoothings $g_{\epsilon}\in C^{\infty}$ such that $g_{\epsilon}\rightarrow g$ in $W^{1,p}_{loc}$ as $\epsilon\rightarrow0$.
And From \cite{Lee2015}, we have
\begin{thm}\label{m}
  Let $g$ be a $W^{2,p}_{-q}$ asymptotically flat metric with $p>n$ and $q\geq \frac{n-2}{2}$, and assume that the scalar curvature of $g$ is integrable. Then the generalized ADM mass coincides with the standard ADM mass.
\end{thm}

\begin{thm}
  Let $M$ be a smooth $n$-manifold with asymptotically flat metric $g\in C^0\bigcap W^{1,p}_{-q}$, $p>n$, $q>\frac{n-2}{2}$. Assume that $g$ has bounded curvature
  $C'\leq Rm(g)\leq C$, in the sense of Aleksandrov and the Ricci curvature distribution $\ll R_{ij},u \gg\in L^p_{-q-2}$.
  If the distributional scalar curvature $\ll R_g,u\gg$ is a finite, signed measure outside some compact set and $q=n-2$,
  then, we have a series of smoothings $g_{\epsilon}\in W^{2,p}_{-q}$, such that
  \[ \liminf_{\epsilon\rightarrow0}m_{ADM}(M,g_{\epsilon})=m_{ADM}(M,g) \]
\end{thm}
\begin{proof}
 From Theorem \ref{gepsilon}, we have $g_{\epsilon}\in W^{2,p}_{-q}$. And by Theorem \ref{bdd}, we have $g_{\epsilon}\in L^1$ because we assumed that $R_g$ is a finite signed measure outside a compact set. Then Theorem \ref{m} implies that the distributional ADM mass of $g_{\epsilon}$ is equal to the usual ADM mass.
\begin{equation*}
\begin{split}
   &m_{ADM}(M,g)-m_{ADM}(M,g_{\epsilon})  \\
    = & \frac{1}{2(n-1)\omega_{n-1}}\inf_{\delta>0}\liminf_{\rho\rightarrow+\infty}(\frac{1}{\delta}\int_{\rho<r<\rho+\delta}V(g)\cdot\bar{\nabla}rd\mu_h)\\
    &-\frac{1}{2(n-1)\omega_{n-1}}\lim_{\rho\rightarrow+\infty}\int_{\rho=r}(g_{\epsilon ij,i}-g_{\epsilon ii,j})\upsilon_jd\sigma_h \\
    = & \frac{1}{2(n-1)\omega_{n-1}}\inf_{\delta>0}\liminf_{\rho\rightarrow+\infty}
    \left(\frac{1}{\delta}\int_{\rho<r<\rho+\delta}(V(g)-V(g_{\epsilon}))\cdot\bar{\nabla}rd\mu_h\right).
\end{split}
\end{equation*}
Now since
\begin{equation*}
\begin{split}
    & |V^k(g)-V^k(g_{\epsilon})| \\
     =& |g^{ij}g^{kl}(g_{il,j}-g_{ij,l})-g^{ij}_{\epsilon}g^{kl}_{\epsilon}(g_{\epsilon il,j}-g_{\epsilon ij,l})| \\
     \leq & |g^{ij}g^{kl}g_{il,j}-g^{ij}_{\epsilon}g^{kl}_{\epsilon}g_{\epsilon il,j}|+|g^{ij}g^{kl}g_{ij,l}-g^{ij}_{\epsilon}g^{kl}_{\epsilon}g_{\epsilon ij,l}| \\
     \leq & |((1+C_1(\epsilon))^3-1)g^{ij}_{\epsilon}g^{kl}_{\epsilon}g_{\epsilon il,j}|+|((1+C_1(\epsilon))^3-1)g^{ij}_{\epsilon}g^{kl}_{\epsilon}g_{\epsilon ij,l}| \\
     \leq & 4C_1(\epsilon)(|g^{ij}_{\epsilon}g^{kl}_{\epsilon}g_{\epsilon il,j}|+|g^{ij}_{\epsilon}g^{kl}_{\epsilon}g_{\epsilon ij,l}|)
\end{split}
\end{equation*}
we have
\begin{equation*}
\begin{split}
   &|m_{ADM}(M,g)-m_{ADM}(M,g_{\epsilon})| \\
   \leq &\frac{1}{2(n-1)\omega_{n-1}}\inf_{\delta>0}\liminf_{\rho\rightarrow+\infty}
    \left(\frac{1}{\delta}\int_{\rho<r<\rho+\delta}|V(g)-V(g_{\epsilon})|d\mu_h\right) \\
    \leq &\frac{1}{2(n-1)\omega_{n-1}}\inf_{\delta>0}\liminf_{\rho\rightarrow+\infty}\left(\frac{1}{\delta}\int_{\rho<r<\rho+\delta}
    4C_1(\epsilon)(|g^{ij}_{\epsilon}g^{kl}_{\epsilon}g_{\epsilon il,j}|+|g^{ij}_{\epsilon}g^{kl}_{\epsilon}g_{\epsilon ij,l}|)d\mu_h\right)\\
    =&\frac{4C_1(\epsilon)}{2(n-1)\omega_{n-1}}\lim_{\rho\rightarrow+\infty}\int_{r=\rho}
    (|g^{ij}_{\epsilon}g^{kl}_{\epsilon}g_{\epsilon il,j}|+|g^{ij}_{\epsilon}g^{kl}_{\epsilon}g_{\epsilon ij,l}|)d\mu_h
\end{split}
\end{equation*}
The last equality follows from the fact that $g_{\epsilon}\in W^{2,p}_{-q}$, then $\partial g_{\epsilon}\in C^{0,\alpha}_{-q}$, which consequences that the integrand of the last equality is Holder continuous.
For $q=n-2$, $\partial g_{\epsilon}=O(r^{-q-1})=O(r^{-n+1})$ implies that the last limit is finite. Thus,
\[ |m_{ADM}(M,g)-m_{ADM}(M,g_{\epsilon})|\leq C(n)C_1(\epsilon) \]
where $C(n)$ is a constant depending on $n$. Then we obtain the result.
\end{proof}

\section{Proof of the main theorem}
From the above discussion, we have shown that

\begin{thm}\label{main}
  Let $M$ be a smooth $n$-manifold $(3\leq n\leq7)$ with an asymptotically flat metric $g\in C^0\bigcap W^{1,p}_{-q}$, $p>n$, $q>\frac{n-2}{2}$. Assume that
  \renewcommand\labelenumi{(\theenumi)}
  \begin{enumerate}
    \item $g$ has bounded curvature $Rm(g)$
    \[ C'\leq Rm(g) \leq C \]
    in the sense of Aleksandrov;
    \item the Ricci curvature distribution $\ll R_{ij},u \gg\in L^p_{-q-2}$;
    \item $g$ has non-negative scalar curvature distribution $R_g$
    \[ \ll R_g,u \gg\geq 0 \]
    for any smooth compactly supported non-negative function $u$ on $M$;
    \item the set $\Omega=\{ x\in M:Rm(g)(x)\leq0 \}$ is compact;
    \item
    \[ \lim_{\delta\rightarrow0}\frac{1}{\delta}\int_{0<l(x)<\delta}V\cdot\bar{\nabla}ld\mu_{h}=0 \]
    where $l(x)=dist(x,\Omega)$, for any smooth background metric $h$, $\cdot$ and $\bar{\nabla}$ are inner product and Levi-Civita connection with respect to $h$, $V$ is as in the definition of scalar curvature distribution;
    \item $R_g$ is a finite, signed measure outside a compact set.
  \end{enumerate}
Then, we have a series of smoothings $g_{\epsilon}\in W^{2,p}_{-q}$ that satisfying
\begin{enumerate}
  \item $\frac{1}{1+C_1(\epsilon)}g\leq g_{\epsilon}\leq (1+C_1(\epsilon))g$, with $\lim_{\epsilon\rightarrow0}C_1(\epsilon)=0$;
  \item $R(g_{\epsilon})\geq B$, with fixed $B<0$;
  \item $\int_{\{ R(g_{\epsilon})<0\}} |R(g_{\epsilon})|d\mu_{g_{\epsilon}}<C(\epsilon)$, with $\lim_{\epsilon\rightarrow0}C(\epsilon)=0$ ;
  \item $R(g_{\epsilon})\in L^1$;
  \item if $q=n-2$, then $\liminf_{\epsilon\rightarrow0}m_{ADM}(M,g_{\epsilon})=m_{ADM}(M,g)$.
\end{enumerate}
\end{thm}

We then follow the approach of \cite{McFeron2012} to prove the positive mass theorem in our case. First, we state the existence of Ricci flow with initial data $g\in C^0$ due to M.Simon \cite{Simon2002}.

\begin{definition}
  Given a constant $1\leq\delta<\infty$, a metric $h$ is $\delta$-fair to $g$ if $h$ is $C^{\infty}$ and there exists a constant $k_0$ with
   \[ \sup_{x\in M}|Rm(h)(x)|_h=k_0<\infty \]
  and \[\frac{1}{\delta}h(p)\leq g(p)\leq\delta h(p)\]
   for all $p\in M$.
\end{definition}

\begin{thm}\label{gt}
  Let $h$ be a metric which is $1+\epsilon(n)$-fair to $g$,then
  \begin{itemize}
    \item $\exists T=T(n,k_0)>0$ and a family of metrics $g(t)$,$t\in(0,T]$ in $C^{\infty}(M\times(0,T])$ which solves the $h$-flow on $(0,T]$;
    \item $h$ is $1+2\epsilon(n)$-fair to $g(t)$ for $t\in(0,T]$;
    \item $g(t)\rightarrow g$ as $t\rightarrow 0$ uniformly on compact sets.
  \end{itemize}
  where the $h$-flow is the Hamilton-Deturk flow with metric $h$,
  \[ \frac{\text{d}}{\text{d}t}g_{ij}=-2R_{ij}+\nabla_i W_j+\nabla_j W_i \]
  \[ W_j=g_{jk}g^{pq}(\Gamma_{pq}^k-\bar{\Gamma}_{pq}^k) \]
  where $\bar{\Gamma}$ are the Christoffel symbols with respect to $h$.
\end{thm}

This flow is equivalent to the Ricci flow under a diffeomorphism and this flow is parabolic but the Ricci flow is not.
The way to construct the solution $g(t)$ is as follows.

\begin{enumerate}[(i)]
  \item take a sequence of smoothings $g_{\epsilon}$ converges to $g$;
  \item for every $\epsilon$, solve the $h$-flow in a short time with initial data $g_{\epsilon}$ to obtain $g_{\epsilon}(t)$;
  \item there exists a $T>0$ independent of $\epsilon$ such that
  \begin{equation}\label{derivative}
    |\bar{\nabla}^kg_{\epsilon}(t)|\leq C_kt^{-\frac{k}{2}}
  \end{equation}
  for any $t\in (0,T]$, here $\bar{\nabla}$ is the Levi-Civita connection with respect to $h$ and $C_k$ is independent of $\epsilon$;
  \item extract a subsequence of $g_{\epsilon}(t)$ as $\epsilon\rightarrow0$ such that $g_{\epsilon}(t)\rightarrow g(t)$ in $C^k$ for any $k$ on compact set.
\end{enumerate}

By Theorem \ref{main}, we know $g_{\epsilon}$ is $(1+C_1(\epsilon))$ fair to $g$ and because the curvature of $g$ is uniformly bounded and thus the curvature of $g_{\epsilon}$ is also uniformly bounded by Theorem \ref{bdd}. Let $g(t)$ be the $g_{\epsilon}$-flow on $(0,T]$ for sufficiently small $\epsilon$ constructed in Theorem \ref{gt}. In fact, we can modify $g_{\epsilon}$ to be equal to the Euclidean metric outside a compact set, and it will also be $1+C(\epsilon)$-fair to $g$. Let $g_{\epsilon}(t)$ be the $g_{\epsilon}$-flow with initial data $g_{\epsilon}$ in a short time.

Now $g_{\epsilon}\in W^{2,p}_{-q}\subset C^{1,\alpha}_{-q}$ for $p>n$, $q>\frac{n-2}{2}$. By Lemma 16 of \cite{McFeron2012}, there is a $T>0$ independent of $\epsilon$,
such that for any $t_0\in(0,T]$, we have $g_{\epsilon}(t)\in C^{1,\alpha}_{-(q-\alpha)}$ for $t\in [t_0,T]$ and $\epsilon>0$. By taking $\epsilon\rightarrow0$, we conclude
that $g(t)\in C^{1,\alpha}_{-(q-\alpha)}$ for $t\in [t_0,T]$. If $q>\frac{n-2}{2}$ and $\alpha$ is sufficiently small, then $g(t)\in C^{1,\alpha}_{-q'}$ with $q'=q-\alpha>\frac{n-2}{2}$. Therefore, $g(t)$ is asymptotically flat in $C^{1,\alpha}_{-q}$ with $q>\frac{n-2}{2}$.

Next, we will check that $g(t)$ satisfies the assumptions of positive mass theorem. The following theorem in \cite{McFeron2012} will be used.
\begin{thm}\label{14}
  Suppose that $g_{\epsilon}\rightarrow g$ locally uniformly in $C^2$ and $g_{\epsilon}\in C^{1,\alpha}_{-q}$ for $q>\frac{n-2}{2}$. In addition we assume that $R(g_{\epsilon})\in L^1$ for all $\epsilon>0$ and the scalar curvature of $g_{\epsilon}$ satisfies
  \begin{equation}\label{non}
    \int_{\{R(g_{\epsilon})<0\}}|R(g_{\epsilon})|dV<\epsilon
  \end{equation}
  Then the scalar curvature $R(g)$ of $g$ is nonnegative and integrable, and
  \[ m(g)\leq\liminf_{\epsilon\rightarrow0}m(g_{\epsilon}) \]
  assuming that the limit is finite.
\end{thm}

Thus, we only have to check that $g_{\epsilon}(t)$ satisfies the assumptions of Theorem \ref{14}. There is some slight difference between our condition and theirs in \cite{McFeron2012}.

\begin{thm}\label{15}
   Let $M$ be an asymptotically flat $n$-manifold with smooth metric $g$. Suppose  $g\in W^{2,p}_{-q}$, $p>n$, $q>\frac{n-2}{2}$, the scalar curvature $R(g)$ is integrable, and $g$ has bounded curvature. Let $g(t)$ be the Ricci flow for $t\in [0,T]$ starting with $g(0)=g$. If
   \[ \int_{M\backslash B_r}|R(g)|dV<\eta(r) \]
   with $\lim_{r\rightarrow\infty}\eta(r)=0$. Then there exists $\tilde{\eta}(r)$ depending not on $t$, such that for $t\in[0,T]$ we have
   \[ \int_{M\backslash B_r}|R(g(t))|dV<\tilde{\eta}(r) \]
   and $\lim_{r\rightarrow\infty}\tilde{\eta}(r)=0$. In particular, $R(g(t))\in L^1$ for $t\in [0,T]$.
\end{thm}
\begin{proof}
  The proof of the theorem is similar to the one in \cite{McFeron2012}. The difference is that we only assume $g\in W^{2,p}_{-q}$ while they supposed that $g\in C^2_{-q}$.
  Thus, we do not have the decay of Ricci curvature, i.e. $|Ric|^2<C|x|^{-2q-4}<C|x|^{-n-2}$, which consequences
   \begin{equation}\label{ric}
     \int_{M\backslash B_{r_1}}|Ric|^2dV<C\int_{M\backslash B_{r_1}}r^{-n-2}dV\leq Cr_1^{-2}
   \end{equation}
    This is used in the proof of \cite{McFeron2012}. We will show that \eqref{ric} is also right in our situation with a bit difference and then the theorem is proved.
    For $g\in W^{2,p}_{-q}$, then the Ricci curvature $Ric$ of $g$ is in $L^p_{-q-2}$. Since $p>n\geq 3$, then $\frac{p}{2}>1$ and by Holder inequality, we have
    \begin{equation*}
    \begin{split}
        & \int_{M\backslash B_{r_1}}|Ric|^2dx \\
         \leq & \left(\int_{M\backslash B_{r_1}}\left( |Ric|^2r^{2(q+2)-\frac{2n}{p}} \right)^{\frac{p}{2}}dx \right)^{\frac{2}{p}}
                \left( \int_{M\backslash B_{r_1}}\left(r^{-2(q+2)+\frac{2n}{p}}\right)^{\frac{p}{p-2}}dx \right)^{\frac{p-2}{p}}  \\
        = & \left( \int_{M\backslash B_{r_1}}|Ric|^pr^{(q+2)p-n}dx \right)^{\frac{2}{p}}
            \left(\int_{M\backslash B_{r_1}}r^{(-2(q+2)+\frac{2n}{p})\frac{p}{p-2}}dx\right)^{\frac{p-2}{p}}  \\
        = & \|Ric\|^2_{L^p_{-q-2}}\cdot \left( \int_{M\backslash B_{r_1}}r^adx \right)^{\frac{p-2}{p}}
    \end{split}
    \end{equation*}
    where $a=(-2(q+2)+\frac{2n}{p})\frac{p}{p-2}$. Since $q>\frac{n-2}{2}$, $p>n$, then we have $2(q+2)>n+2$. Therefore,
    \begin{equation*}
      \begin{split}
         a= & (-2(q+2)+\frac{2n}{p})\frac{p}{p-2} \\
           =& -2(q+2)\left(1+\frac{2}{p-2}\right)+\frac{2n}{p-2} \\
           =& -2(q+2)+\frac{2}{p-2}(n-2(q+2)) \\
           <& -(n+2)+\frac{2}{p-2}(n-(n+2))\\
           =& -(n+2)-\frac{4}{p-2}<-n-2
      \end{split}
    \end{equation*}
    As a result,
    \[ \int_{M\backslash B_{r_1}}|Ric|^2dx<C\left(\int_{M\backslash B_{r_1}}r^{-n-2}dx\right)^\frac{p-2}{p}=Cr_1^{-2\frac{p-2}{p}} \]
    and $-2\frac{p-2}{2}=-2+\frac{4}{p}\leq -2+\frac{4}{3}<0$. This is enough for proving the theorem.
\end{proof}

Apply Theorem \ref{15} to $g_{\epsilon}$, we obtain $R(g_{\epsilon}(t))\in L^1$ for $t\in [0,T]$ and for every $\epsilon>0$. To check the condition \eqref{non}, we need the next theorem.

\begin{thm}\label{16}
  Let $M$ be an asymptotically flat $n$-manifold with smooth metric $g$. Suppose  $g\in W^{2,p}_{-q}$, $p>n$, $q>\frac{n-2}{2}$, the scalar curvature $R(g)$ is integrable, and $g$ has bounded curvature. Let $g(t)$ be the Ricci flow for $t\in [0,T]$ starting with $g(0)=g$. For any $t>0$ we have along the Ricci flow
  \[ \lim_{r\rightarrow\infty}\int_{\partial B_r}|\nabla R|dS=0 \]
  Moreover, if $t_0>0$ then this convergence is uniform for $t\in [t_0,T]$.
\end{thm}
\begin{proof}
   We know the evolution equation of the scalar curvature $R(g(t))(x)$ under the Ricci flow is
   \[ \frac{\partial R}{\partial t}=\Delta R+2|Ric|^2 \]
   In addition we will work outside the compact set $B$ where the asymptotic coordinates are defined. As long as $t\in [\frac{t_0}{2},T]$, the metrics along the flow are uniformly equivalent to the Euclidean metric outside $B$ and also their derivatives are controlled. Thus the evolution equation is uniformly parabolic.
   Let $\tau\in [t_0,T]$, $p\in M\backslash B$, and choose $r$ with $0<r\leq \sqrt{\tau-\frac{t_0}{2}}$. Define the parabolic cylinder
   \[ Q(r)=\{ (x,t):|x-p|<r, \tau-r^2<t<\tau \} \]
   By the local maximum principle for parabolic equations (Theorem 8.1 in LSU\cite{Ladyzenskaja1968}), we get
   \[ \sup_{Q(\frac{r}{2})}|R|\leq C(r)\left(\|R\|_{L^2(Q(r))}+\|(|Ric|^2)\|_{L^s(Q(r))}\right) \]
   with $s>\frac{n}{2}$. By the derivative estimate (Theorem 11.1 in LSU\cite{Ladyzenskaja1968}), we obtain
   \[ \sup_{Q(\frac{r}{4})}|\nabla R|\leq C'(r)\left(\sup_{Q(\frac{r}{2})}|R|+\|(|Ric|^2)\|_{L^s(Q(\frac{r}{2}))}\right) \]
   Thus,
   \[ \sup_{Q(\frac{r}{4})}|\nabla R|\leq C(r)\left(\|R\|_{L^2(Q(r))}+\|(|Ric|^2)\|_{L^s(Q(r))}\right)\]
   Now we cover the sphere $\partial B_a$ with balls $B_i$ of radius $\frac{r}{4}$. Each point can be covered by at most $c(n)$ of balls $4B_i$, where $c(n)$ depends on the dimension. Then we apply the last inequality and integrate to get
   \begin{equation}\label{d}
     \int_{\partial B_a}|\nabla R|dS\leq C(n,r)\left( \|R\|_{L^2(A(a,r))}+\|(|Ric|^2)\|_{L^s(A(a,r))} \right)
   \end{equation}
   where $A(a,r)=\{ (x,t):d(x,\partial B_a)<r, \tau-r^2<t<\tau \}$.

   Since $g\in W^{2,p}_{-q}$, then $R\in L^p_{-q-2}\subset L^p_{-q-1}\subset L^2_{-\frac{n}{2}}=L^2$. Since
   \[ \|(|Ric|^2)\|_{L^s(A(a,r))}=\left( \int_{\tau-r^2}^{\tau}\int_{B_{a+r}\backslash B_{a-r}}|Ric|^{2s}dxdt \right)^{\frac{1}{s}} \]
   and $Ric\in L^p_{-q-2}$, by $p>n$, $q>\frac{n-2}{2}$, we can choose $n<2s<p$,  then
   \begin{equation*}
     \begin{split}
        &\int_{B_{a+r}\backslash B_{a-r}}|Ric|^{2s}dx  \\
          \leq& \left( \int_{B_{a+r}\backslash B_{a-r}}\left(|Ric|^{2s}\rho^{((q+2)p-n)\frac{2s}{p}} \right)^{\frac{p}{2s}}dx\right)^{\frac{2s}{p}}
          \left( \int_{B_{a+r}\backslash B_{a-r}}\left( \rho^{(n-(q+2)p)\frac{2s}{p}} \right)^{\frac{p}{p-2s}}dx \right)^{\frac{p-2s}{p}} \\
         \leq & \|Ric\|^{2s}_{L^p_{-q-2}}\cdot\left( \int_{B_{a+r}\backslash B_{a-r}}\rho^zdx \right)^{\frac{p-2s}{p}}
     \end{split}
   \end{equation*}
   where $z=(n-(q+2)p)\frac{2s}{p}\cdot\frac{p}{p-2s}$. Since
   \begin{equation*}
     \begin{split}
        z= & (n-(q+2)p)\frac{2s}{p}\cdot\frac{p}{p-2s} \\
          =& (\frac{2ns}{p}-2s(q+2))(1+\frac{2s}{p-2s}) \\
         = & -2s(q+2)+\frac{2s}{p-2s}(\frac{2ns}{p}-2s(q+2)+\frac{n(p-2s)}{p}) \\
         <& -s(n+2)+\frac{2s}{p-2s}(n-s(n+2))\\
         <& -\frac{n(n+2)}{2}+\frac{2s}{p-2s}(n-\frac{n(n+2)}{2})\\
         <& -\frac{n(n+2)}{2}<-(n+2)
     \end{split}
   \end{equation*}
   thus,
   \begin{equation*}
     \begin{split}
        \left( \int_{B_{a+r}\backslash B_{a-r}}\rho^zdx \right)^{\frac{p-2s}{p}}\leq &\left( \int_{B_{a+r}\backslash B_{a-r}}\rho^{-n-2}dx \right)^{\frac{p-2s}{p}}  \\
         \leq & C(n)\left( \int_{a-r}^{a+r}\rho^{-3}d\rho \right)^{\frac{p-2s}{p}} \\
          =& C(n)\left( (a-r)^{-2}-(a+r)^{-2} \right)^{\frac{p-2s}{p}} \\
         = & C(n)\left(\frac{4ar}{(a^2-r^2)^2}\right)^{\frac{p-2s}{p}}.
     \end{split}
   \end{equation*}
   Let $a\rightarrow\infty$ in \eqref{d}, we then get the result.
\end{proof}

\textbf{Remark:} In \cite{McFeron2012}, they prove Theorem \ref{16} by using the $L^{\infty}$ norm of $|Ric|^2$ for the reason that $|Ric|^2$ has decay rate less than $-n-2$. However, we assumed $Ric\in L^p_{-q-2}$, which is not sufficient if we still use $L^{\infty}$ norm of $|Ric|^2$. That is why we have to use the $L^{s}$ norm of $|Ric|^2$  with $s>\frac{n}{2}$.

Because of Theorem \ref{16}, Lemma 13 of \cite{McFeron2012} is valid in our situation, which implies \eqref{non} for $g_{\epsilon}(t)$. Now by Simon's result \eqref{derivative} and Theorem \ref{14}, we obtain $R(g(t))$ is nonnegative and integrable for $t>0$ and $m(g(t))\leq \liminf_{\epsilon\rightarrow0}m(g_{\epsilon}(t))$.
So Theorem \ref{massthm} implies $m(g(t))\geq0$.
Similarly, the ADM mass is still preserved under the Ricci flow in our case. This is easy to be proved by the way of \cite{McFeron2012} as long as we have the following
\begin{equation*}
  \begin{split}
     &\int_{\partial B_a}r^{-q-1}|Ric|dS \\
       \leq&  \left( \int_{\partial B_a}|Ric|^{p}r^{(q+2-\frac{n-1}{p})p}dS \right)^{\frac{1}{p}}\left(\int_{\partial B_a}r^{z}dS\right)^{\frac{p-1}{p}} \\
       =& \|Ric\|_{L^p_{-q-2}(\partial B_a)}\cdot\left(\int_{\partial B_a}r^zdS\right)^{\frac{p-1}{p}}
  \end{split}
\end{equation*}
where $z=(\frac{n-1}{p}-2q-3)\frac{p}{p-1}<-n-1$.

Finally, we can prove the positive mass theorem in our condition.

\begin{thm}
  Under the assumptions of Theorem \ref{main}, if $q=n-2$, we have $m_{ADM}(g)\geq 0$. Moreover, $m_{ADM}(g)=0$ if and only if $g$ is diffeomorphic to the Euclidean metric.
\end{thm}
\begin{proof}
   We have shown
   \[ 0\leq m(g(t))\leq \liminf_{\epsilon\rightarrow0}m(g_{\epsilon}(t))= \liminf_{\epsilon\rightarrow0}m(g_{\epsilon})\]
   and from Theorem \ref{main},
   \[ \liminf_{\epsilon\rightarrow0}m(g_{\epsilon})=m_{ADM}(g) \]
   then we get
   \[ m_{ADM}(g)\geq0. \]
   Now we suppose $m_{ADM}(g)=0$, then $m(g(t))=0$ for $t>0$. By theorem \ref{massthm}, $(M,g(t))\cong(\mathbb{R}^n,\delta)$. Then the rigidity part of the theorem follows from Theorem 18 of \cite{McFeron2012}.
\end{proof}

\bibliography{ricci-flow-positive-mass}

\end{document}